\documentclass[12pt]{amsart}
\usepackage{amsmath, amssymb}
\usepackage{amssymb}
\usepackage{latexsym}
\usepackage{amscd}
\usepackage{amsthm}
\usepackage{graphicx}

\theoremstyle{plain}
\newtheorem{thm}{Theorem}
\newtheorem{cor}{Corollary}
\newtheorem{lem}{Lemma}
\newtheorem{prop}{Proposition}

\theoremstyle{definition}

\newcommand{\inter}{\textrm{int}}

\newcommand{\T}{\mathbb{T}}

\newcommand{\R}{\mathbb{R}}

\newcommand{\eps}{\varepsilon}

\begin{document}
\title{Analytic skew products of quadratic polynomials over circle expanding maps}
\author{Wen Huang, Weixiao Shen}
\address{Department of Mathematics, University of Science and Technology of China,
Hefei, Anhui, 230026, P.R. China}
\address{Department of
Mathematics, National University of Singapore, Singapore}
\email{wenh@mail.ustc.edu.cn,\, matsw@nus.edu.sg}

\subjclass[2000]{37C40,34C28,34D08}
\maketitle
\begin{abstract}
We prove that a Viana map with an arbitrarily non-constant real
analytic coupling function admits two positive Lyapunov exponents
almost everywhere.
\end{abstract}
\section{Introduction}
We study skew products of quadratic maps driven by expanding circle
maps with real-analytic coupling functions.
 Let $f(x)=a-x^2$,
$a\in (1,2)$ be a quadratic map for which $x=0$ is a strictly
pre-periodic point, let $d\ge 2$ be an integer and let $\phi:\T\to
\R$ be a non-constant real-analytic function, where
$\T=\mathbb{R}/\mathbb{Z}$.
For $\alpha>0$, define $F=F_\alpha: \T\times\R \to \T\times\R$ as
$$F(\theta, x)= (d\theta, f(x)+\alpha\phi(\theta)).$$

These kind of maps, nowadays called {\em Viana maps}, were first
studied in \cite{V}, where Viana proved that for $d\ge 16,$
$\phi(\theta)=\sin (2\pi \theta)$, if $\alpha$ is small enough, then
$F_\alpha$ has two positive Lyapunov exponents  Lebesgue almost
everywhere. Furthermore, he showed that the same conclusion holds
for any small $C^3$ perturbation of $F$. In \cite{BST},
Buzzi-Sester-Tsujii extended the result to the case $d\ge 2$ (for
the same coupling function $\phi$). The non-integer case of $d$ was considered in \cite{Sch}.
See also~\cite{Sch2, GS} for results on skew products driven by certain quadratic polynomials. These maps serve as typical examples of multi-dimensional non-uniformly
expanding dynamical systems for which a general theory have been
developed, see the excellent review \cite{Alves}.

An important feature of the maps $F$ considered in the papers cited
above is that they display partial hyperbolicity: the maps are
uniformly expanding in the horizontal direction and the horizontal
expansion dominates the vertical expansion. The property implies
that images of horizontal curves are nearly horizontal. The
particular form of the coupling function $\phi$ allows the authors
to conclude that high iterates of a horizontal curve are non-flat,
which is an important technical point in the proof.

The goal of this paper is to weaken the assumption on the coupling
function.
\begin{thm}\label{thm:main} Fix $f, g, \phi$ as above.
If $\alpha$ is small enough, then  $F$ has two positive Lyapunov
exponents at Lebesgue almost every point in $\T\times \R$. Moreover,
the same holds for any small $C^\infty$ perturbation of $F$.
\end{thm}

We shall only deal with the unperturbed case, as the perturbed case
follows by the same strategy in \cite{V}. As in the previous works,
we analyze return of an ``{\em admissible curve}'' (or {\em a
sufficiently high iterate of a horizontal curve}) to small
neighborhoods of the critical circle $x=0$. The difference here from
\cite{BST} is that for general $\phi$, the images of an admissible
curve are not necessarily separated from each other on a definite
part, so that a key estimate (Lemma 2.6 in \cite{V} and Proposition
5.2 in \cite{BST}) only holds in a weaker form. Nevertheless, we
shall prove that most points in an admissible curve cannot return
too close to the critical circle before displaying some (weak)
expansion, see Proposition~\ref{prop:return}. The real-analyticity
of the coupling function is essentially used in
\S~\ref{subsec:comega} to obtain non-flatness of admissible curves.
It is not clear to us whether the same result holds for smooth
coupling functions $\phi$ with non-flat critical points.

Let us mention a few consequences of our theorem. Since we do not
obtain exactly the same underlying estimates as in \cite{V}, the
proof in \cite{A} does not apply directly in our case to show
existence of SRB measures. However, our result shows that $F$ (and
any small $C^\infty$ perturbation) satisfies the assumption of
\cite{ABV}, from which we conclude that $F$ admits finitely many
acip's. Moreover, a slight modification of \cite{AV} shows that $F$
is ergodic with respect to the Lebesgue
measure. Thus, in our case, $F$ and its small $C^\infty$
perturbation admit a unique acip.

{\em Note.} Let $f_n:\T\times \mathbb{R}\to \mathbb{R}$ be defined
so that
$$F^n(\theta, x)=(d^n \theta, f_n(\theta, x)).$$

Choose $\beta$ slightly smaller than $|p_1|$, where
$p_1=\frac{-1-\sqrt{1+4a}}{2}$ is the orientation-preserving fixed
point of $f$. Then $1<\beta < 2$ and the interval $B=[-\beta,
\beta]$ satisfies: $f(B)\subset \inter (B)$ and $|f_a'(x)|\ge
2\beta$ for $x\in \mathbb{R}\setminus \text{int}(B)$. Let
$I=\T\times B$. Then provided that $\alpha>0$ is small enough, the
following hold:
\begin{itemize}

\item $F(I)\subset \inter (I)$.

\item $|\frac{ \partial f_1(\theta,x)}{\partial x}|>\beta>1$ outside of $I$.

\item $|\frac{ \partial f_1(\theta,x)}{\partial x}|\le 2\beta\le 4$
on $I$.
\end{itemize}

{\em Note.} Unless otherwise stated, all constants appearing below
depend on $f$, $d$ and $\phi$. Dependence of constants on $\alpha$
will be stated explicitly.

Without loss of generality, we assume $|\phi(\theta)|\le 1$ for all
$\theta\in \T$.
\section{Preliminaries}
\subsection{Domination}
An important feature of the map $F$ is that the horizontal expansion
dominates the vertical expansion.

\begin{lem}\label{lem:lyapunov}
There exists $C_1>0$,  $R_1\in (0,2)$ and $n_0\in \mathbb{N}$ such
that, when $\alpha$ is sufficiently small, for all $(\theta,x)\in
I$, we have
$$\left|\frac{\partial f_k(\theta,x)}{\partial x}\right|\le
C_1 R_1^k \text{ for }k\in \mathbb{N} \text{ and }
\left|\frac{\partial f_k(\theta,x)}{\partial x}\right|\le  R_1^k \le
\frac{d^k}{6}\text{ for }k\ge n_0.$$
\end{lem}
\begin{proof}  Following from the proof of Lemma 3.1 in \cite{BST}, one can show that there exists $0<C_1<\infty$ and $R_*\in (0,2)$ such that
when $\alpha$ is sufficiently small, for all $(\theta,x)\in I$, we
have
$$\left|\frac{\partial f_k(\theta,x)}{\partial x}\right|\le C_1 R_*^k \text{ for }k\in
\mathbb{N}.$$ Then we take $R_1\in (R_*,2)$ and $n_0\in \mathbb{N}$
satisfying $C_1R_*^{n_0}\le R_1^{n_0}\le \frac{d^{n_0}}{6}$. This
gives us the desired estimate.
\end{proof}

\subsection{Building expansion}
The following is a variation of Lemma 2.5 in \cite{V}. The same
proof works.

Given $(\theta, x)\in \T\times B$, write $(\theta_j, x_j)=
F^j(\theta, x)$.

\begin{lem}\label{lem:vianaexpansion}
There are $\lambda>1$, $K>0$ and $\delta>0$ such that, when $\alpha$ is sufficiently small,
$$\left|\frac{\partial f_n(\theta, x)}{\partial x}\right|\ge
K\alpha^{0.6} \lambda^n$$ for all $(\theta, x)\in \T\times B$ with
$|x_0|, |x_1|, \ldots, |x_{n-1}|\ge \alpha^{0.6}$. If, in addition,
$|x_n|<\delta$ then we even have
$$\left|\frac{\partial f_n(\theta, x)}{\partial x}\right|\ge
K\lambda^n.$$
\end{lem}

\subsection{Shadowing}\label{sec:shadowing}
We present a more precise version of Lemma 2.4 in \cite{V} here.

Fix a constant $\xi_0>0$ such that for all $i\ge 1$,
$$c_i:=f^i(0)\not \in (-2\xi_0, 2\xi_0).$$

Assume $\alpha\in (0, \xi_0]$. Define $I_1=[c_1-2\alpha,
c_1+2\alpha]$, and for $n\ge 1$, inductively define
$I_{n+1}$ be the closed interval with the following property:
$I_{n+1}\Supset f(I_n)$ and both components of $I_{n+1}\setminus
f(I_n)$ have length $\alpha$. Let $N(\alpha)$ be the maximal
integer such that the following hold:
\begin{itemize}
\item $|I_n|<\xi_0$ for all $n=1,2,\ldots, N(\alpha)$;
\item $\sum_{n=1}^{N(\alpha)} |I_n|\le \xi_0.$
\end{itemize}
Note that $|\phi|\le 1.$ Then $F: \T\times I_n\to \T\times I_{n+1}$
is fibre-wise diffeomorphic, for any $n=1,2,\ldots,
N(\alpha)$.
\begin{lem}\label{lem:distortion}
For all $1\le m<
n\le N(\alpha)+1$, and for any $x_i\in I_i$, $i=m,m+1,
\ldots, n-1$, we have
$$e^{-1} |(f^{n-m})'(c_m)| \le \prod_{i=m}^{n-1}|f'(x_i)|\le e |(f^{n-m})'(c_m)|.$$
\end{lem}
\begin{proof} In fact, $|c_i-x_i|\le |I_i|$ for all $m\le i<n$. Thus
\begin{align*}
\left|\log
\frac{\prod_{i=m}^{n-1}|f'(x_i)|}{|(f^{n-m})'(c_m)|}\right| &=
\sum_{i=m}^{n-1} \log \left|1+ \frac{x_i-c_i}{c_i}\right|  \le
\sum_{i=m}^{n-1} \frac{|x_i-c_i|}{|c_i|}\\ & \le
\xi_0^{-1}\sum_{i=m}^{n-1}|I_i| \le 1.
\end{align*}
The statement follows.
\end{proof}

\begin{lem}\label{lem:Nsigma}
There exists $C_0\in (0,1)$ such that
$$C_0\le |(f^{N(\alpha)})'(c_1)|\alpha\le \frac{1}{C_0}.$$
\end{lem}
\begin{proof}
By the mean value theorem, for each $n$, there exists
$x_n\in I_n$ such that
$$|I_{n+1}|=D_n |I_n|+2\alpha,$$
where $D_n=|f'(x_n)|$. It follows that for all $0\le n\le
N(\alpha)$,
\begin{align*}
|I_{n+1}|& = D_n|I_n|+2\alpha\\
& = D_n D_{n-1} |I_{n-1}|+2\alpha (1+D_{n})\\
& = \cdots\\
&= D_n D_{n-1}\ldots D_1 |I_1|+2\alpha (1+D_{n}+ D_{n}
D_{n-1}+\ldots +D_{n} D_{n-1} \ldots D_2)\\
& = 2\prod_{i=1}^n D_i \left(2\alpha +\alpha \sum_{k=1}^n
\prod_{i=1}^k D_i^{-1}\right).
\end{align*}

By Lemma~\ref{lem:distortion}, $\prod_{i=1}^k D_i^{-1}\asymp
|(f^k)'(c_1)|^{-1}$ is exponentially small in $k$. So
$$|I_{n+1}|\asymp |(f^n)'(c_1)|\alpha.$$ Since
$\sum_{n=1}^{N(\alpha)}|(f^n)'(c_1)|/|(f^{N(\alpha)})'(c_1)|$
is bounded from above by a constant depending only on $f$, the
conclusion follows by the definition of $N(\alpha)$.
\end{proof}
\begin{lem}\label{lem:expalpha}
The following holds provided that $\alpha$ is small enough: for each
$(\theta, x)$ with $|x|<\sqrt{\alpha}$,
$$\left|\frac{\partial f_{N(\alpha)}(f_1(\theta, x))}{\partial x}
\right|\ge \frac{C_0}{\alpha},$$ where $C_0>0$ is a universal
constant.
\end{lem}
\begin{proof} Note that
$f_1(\theta, x)\in I_1$. Thus there exists $x_i\in I_i$,
$i=1,2,\ldots, N(\alpha)$ such that
$$\left|\frac{\partial
f_{N(\alpha)}(f_1(\theta, x))}{\partial x}
\right|=\prod_{i=1}^{N(\alpha)}|f'(x_i)|.$$ Applying
Lemmas~\ref{lem:distortion} and~\ref{lem:Nsigma} gives us the result
by redefining $C_0$.
\end{proof}
\section{Admissible curves}
\label{sec:admissible} We prove that the images of horizontal curves
under a sufficiently high iterate of $F$, which will be called {\em
admissible curves}, are nearly horizontal and non-flat.

\subsection{A class of functions.}
\label{subsec:comega}

We introduce a special class of functions which will be needed in
the argument below.

Let $\hat\phi(\theta)=\phi (\theta\mod 1)$. Then $\hat\phi$ is a
real-analytic map from $\R$ to $\R$ with period $1$:
$\hat\phi(\theta+1)=\hat\phi(\theta)$. We say that a function
$T:\R\to \R$ is in {\em the class $\mathcal{T}_{\phi}$} if there
exist $a_n\in \R,$ and $k_n\in \{0,1,\ldots, d^n-1\}$,
$n=1,2,\ldots$ such that the following hold:
\begin{itemize}
\item $T(\theta)=\hat\phi'(\frac{\theta+k_1}{d})+\sum_{n=1}^\infty a_n \hat\phi'(\frac{\theta+k_n}{d^{n+1}})$;
\item $|a_n|\le C_1(\frac{R_1}{ d})^{n}$,
\end{itemize}
where $C_1$ and $R_1$ are as in Lemma \ref{lem:lyapunov}. Take $\rho>0$ such that $\hat\phi$ extends to a holomorphic function
defined in $S=\R\times (-2\rho, 2\rho)$. Then each
$T\in\mathcal{T}_\phi$ may be viewed as holomorphic mappings defined
on $S$. Clearly, $\mathcal{T}_\phi$ is a compact family with respect
to the topology defined by locally uniformly convergence in $S$.
\begin{lem}
\label{lem:classT} There exists $l_0\ge 2$, $A>\mu
>0$ depending only on $\phi$ such that for each $T\in
\mathcal{T}_\phi$, and any $\theta\in [0,1]$,
\begin{equation}\label{eqn:T}
\frac{Ad}{2} \ge \sum_{i=0}^{l_0+1}
|T^{(i)}(\theta)|\ge\sum_{i=0}^{l_0}|T^{(i)}(\theta)|\ge 2 d \mu.
\end{equation}
Moreover,
\begin{equation}\label{eqn:T_maximum}
\sup_{\theta\in [0,1]} |T(\theta)|\ge 2\mu.
\end{equation}
\end{lem}
\begin{proof}
By compactness of the family $\mathcal{T}_{\phi}$, it suffices to
show that $0\not\in\mathcal{T}_{\phi}$. To this end, we first
observe that for any $l\in\mathbb{N}$, $M_l=\sup |\hat\phi^{(l)}|>0$.
Choose $l$ sufficiently large such that
$$ C_1 \frac{ R_1d^{-l}}{1-R_1 d^{-l}}\le 1/2.$$ Now take
$T\in\mathcal{T}_\phi$.  Choose $\theta_0\in [0,1)$ such that
$|\hat\phi^{(l)}(\theta_0)|=M_{l}$,  and we have
\begin{align*}
d^{l-1}\left|T^{(l-1)}(d\theta_0-k_1)\right| & \ge  M_{l}- M_{l} C_1
\sum_{n=1}^{\infty}
\left(\frac{R_1}{d^{l}}\right)^{n}\\
& \ge  M_{l} \left (1-C_1 \frac{ R_1d^{-l}}{1-R_1 d^{-l}}\right)  \\
&\ge \frac{1}{2}  M_{l}.
\end{align*}
This shows that $T$ is not identically zero, completing the proof.
\end{proof}

\subsection{Partial derivatives of $F$}
Now let $d$, $\phi$, $\alpha$, $F$ be as in Theorem~\ref{thm:main},
and let $g(\theta)=d\theta\mod 1$. Recall that $f_n$ is the function
for which
$$F^n(\theta, x)=(g^n(\theta), f_n(\theta, x)).$$
Write
$$H_n(\theta, x)=\frac{\partial f_n(\theta,x)}{\partial \theta}
\text{ and }  V_n(\theta, x)=\frac{\partial f_n(\theta,x)}{\partial
x}.$$ Then
\begin{equation}\label{eqn:vn}
V_n(\theta, x)=  (f^n)'(x)+\sum_{m=1}^{2^{n-1}} \alpha^m
P_{m,n}(\theta,x),
\end{equation}
\begin{equation}\label{eqn:hn}
H_n(\theta, x)=\alpha d^{n-1} G_n(\theta, x)+ \sum_{2\le m\le 2^{n-1}}
\alpha^m Q_{m,n}(\theta, x),
\end{equation}
where $P_{m,n}, Q_{m,n}$ are real analytic functions defined on
$\T\times \R$, and
\begin{equation}\label{eqn:G}
G_n(\theta, x)=\sum_{k=1}^n \frac{(f^{n-k})'(f^kx)}{d^{n-k}}
\phi'(g^{k-1}\theta).
\end{equation}

We identify $\T$ with $[0,1)$ and let $\mathcal{P}_n$ be the
partition of $\T$ consisting of intervals $[j/d^n, (j+1)/d^n)$,
$j=0,1,\ldots, d^n-1$.

\noindent {\em Remark.} For each $\omega\in\mathcal{P}_n$ and $x\in
B$, it is easily seen that there exists $T_{\omega, x}\in
\mathcal{T}_{\phi}$ such that
$$G_n(\theta, x)= T_{\omega, x} (g^n\theta)\mbox{ for all }\theta\in\omega.$$

For any curve $X: \omega\to \R$, $\omega\in\mathcal{P}_n$, let
$F^n(X|\omega)$ denote the curve $Y: [0,1)\to \R$ with
$$Y(\theta)=f_n(\tau(\theta), X(\tau(\theta))),$$ where
$\tau=(g^n|_\omega)^{-1}$.

Let $\mathcal{H}_C(\alpha)$ denote the set of smooth curves
$X:[0,1)\to B$ for which there exists $T\in\mathcal{T}_{\phi}$ such
that $X'-\alpha d^{-1} T$ has $C^{l_0}$ norm bounded from above by
$C\alpha^2$. Let $n_0$ be the positive integer specified in Lemma
\ref{lem:lyapunov}.

\begin{lem}\label{lem:horizontal} There exists a constant $C>0$
such that the following holds provided that $\alpha>0$ is small
enough. Let $X$ be a curve in the class $\mathcal{H}_{2C}(\alpha)$,
$\omega\in \mathcal{P}_{n}$ with $n\ge n_0$, and let $Y=
F^n(X|\omega)$. Then $Y\in \mathcal{H}_C(\alpha)$.
\end{lem}

\begin{proof} Let $C>0$ be a large constant to be determined.

Assume first that $n_0\le n\le 2n_0$. By assumption on $X$, there
exists $T\in\mathcal{T}_{\phi}$ such that the $C^{l_0}$-norm of
$X'-\alpha d^{-1}T$ is bounded from above by $2C\alpha^2$. Let
$$S(\theta)=\sum_{k=1}^n \frac{(f^{n-k})'(x)}{d^{n-k}}\phi'(g^{k-1}\circ\tau(\theta)),$$
and
$$\hat T(\theta) := S(\theta) +
\frac{(f^n)'(x)}{d^n}T(\tau\theta),$$ where $\tau=(g^n|\omega)^{-1}$
and $x=X(\tau(0))$. Then $S, \hat T\in\mathcal{T}_{\phi}$ by Lemma
\ref{lem:lyapunov}. Recall that
$$Y(\theta)= f_n(\tau(\theta), X(\tau(\theta))),$$ hence
\begin{equation*}
Y'(\theta) = H_n(\tau\theta, X(\tau \theta))\tau'(\theta) +
V_n(\tau\theta, X(\tau \theta)) X'(\tau(\theta))\tau'(\theta).
\end{equation*}
Let us estimate
$$Y'(\theta)-\alpha d^{-1} \hat T(\theta)= Z_1(\theta)+Z_2(\theta)
+Z_3(\theta),$$ where
\begin{align*}
Z_1(\theta)&:=H_n(\tau\theta, X(\tau\theta))\tau'(\theta)-\alpha
\frac{1}{d} S(\theta)\\
Z_2(\theta)& :=\frac{V_n(\tau\theta, X(\tau\theta))-
(f^n)'(x)}{d^n}X'(\tau\theta),\\
Z_3(\theta)& :=\frac{(f^n)'(x)}{d^n} \left(X'(\tau\theta)-\alpha
d^{-1} T(\tau\theta)\right).
\end{align*}
By (\ref{eqn:hn}), (\ref{eqn:G}) and (\ref{eqn:vn}), we see that
there exists a constant $K_n$ such that for $i=1,2$,
$$\|Z_i\|_{l_0}\le K_n\alpha\cdot \| X(\tau(\theta))-x\|_{l_0}.$$
Let $$C= 3A\max_{n_0\le n \le 2n_0} K_n,$$ where $A$ is as in
Lemma~\ref{lem:classT}. Then provided that $\alpha$ is small enough,
$$\|X(\tau(\theta))-x\|_{l_0}\le \|X'\|_{l_0}\le \alpha d^{-1} \|T\|_{l_0}+2C\alpha^2
\le \alpha A,$$ hence
$$\|Z_i\|_{l_0}\le C\alpha^2/3, \,\,\,\, i=1,2.$$
By choice of $n_0$, we have $$\|Z_3\|_{l_0}\le C\alpha^2/3.$$ Thus
$$\|Y'-\alpha d^{-1}\hat T\|_{l_0}\le C\alpha^2,$$
completing the proof for the case $n_0\le n\le 2n_0$.

The general case follows by induction. Assume that the conclusion
holds for $n\le kn_0$, $k\ge 2$. To deal with case  $kn_0<n\le (k+1)
n_0$, let $\omega_1$ be the element of $\mathcal{P}_{n-n_0}$ which
contains $\omega$ and let $Y_1:= F^{n-n_0}(X|\omega_1)$. Then by
induction hypothesis $Y_1\in \mathcal{H}_C(\alpha)$. Since
$Y=F^{n_0}(Y_1|g^{n-n_0}(\omega))$, we obtain
$Y\in\mathcal{H}_C(\alpha)$.
\end{proof}

Let us say that a curve $X: [0,1)\to B$ is {\em admissible} if it is
the image of a horizontal curve under $F^{n}$ for some $n\ge n_0+1$:
there exists $\omega\in\mathcal{P}_n$, $x_0\in B$ such that
$X(g^n(\theta))=f_n(\theta,x_0)$ for all $\theta\in \omega$. By
Lemma~\ref{lem:horizontal} and Remark on page 5, $X\in\mathcal{H}_C(\alpha)$.

\begin{prop} \label{prop:deepreturn}
There exists $\eps_0>0$ and $\kappa>0$ such that when $\alpha$ is
sufficiently small, for any admissible curve $X: [0,1)\to B$, and
any $0\le \eps<\eps_0$, we have
$$|\{\theta\in [0,1): |X(\theta)|<\alpha\eps\}|\le \eps ^{\kappa}.$$
\end{prop}
\begin{proof} By definition, there exists $T\in\mathcal{T}_{\phi}$
such that $X'-\alpha d^{-1} T$ has $C^{l_0}$-norm bounded from above
by $C\alpha^2$. Then provided that $\alpha$ is small enough, by
Lemma~\ref{lem:classT}  for any $\theta\in [0,1)$,
$$\mu \alpha\le 2\mu \alpha-C\alpha^2\le \sum_{i=1}^{l_0} |X^{(i)}(\theta)|\le \sum_{i=1}^{l_0+1} |X^{(i)}(\theta)|\le \frac{A}{2}\alpha+C\alpha^2\le A\alpha.$$
It follows that we can divide $[0,1)$ as the disjoint union of
intervals $J_i$, $i=1,2,\ldots, m$, with the following properties:
\begin{itemize}
\item $|J_i|\ge \frac{\mu }{2A l_0}$,
\item there exists $j_i\in
\{1,2,\ldots , l_0\}$ such that $|X^{(j_i)}(\theta)|\ge \frac{\mu
\alpha}{2l_0}$ for $\theta\in J_i$.
\end{itemize}
By Lemma 5.3 in \cite{BST}, $|\{ \theta\in J_i: |X(\theta)|< \alpha
\eps\}|<2^{j_i+1}(\frac{2l_0\eps}{\mu})^{\frac{1}{j_i}}$ for all
$i\in \{ 1,2,\cdots,m\}$ and $\eps>0$. Take $W=\max
\{2^{j_i+1}(\frac{2l_0}{\mu})^{\frac{1}{j_i}}:i=1,2,\cdots,l_0\}$.
Thus
$$|\{ \theta\in [0,1) :X(\theta)|<\alpha
\eps\}|<mW\epsilon^{\frac{1}{l_0}}\le \frac{2Al_0}{\mu}W
\eps^{\frac{1}{l_0}} \text{ for all }\eps \in (0,1). $$ Choosing
$\kappa=\frac{1}{2l_0}$ and $\eps_0\in (0,1)$ with
$\frac{2Al_0}{\mu}W \eps_0^{\frac{1}{2l_0}}\le 1$, the proposition
follows.
\end{proof}

\section{Recurrence to small neighborhoods of the critical circle}
\label{sec:recurrence} We study the recurrence of an admissible
curve $X:[0,1)\to B_\alpha:=[-2\alpha^{0.6}, 2\alpha^{0.6}]$ to the
region $\T\times B_\alpha$. More, precisely, for each $\theta\in
[0,1)$, let $n_0(\theta)=0$ and let
$$n_1(\theta)< n_2(\theta)< \cdots$$
be all the positive integers (finitely or infinitely many) for which
there exists $\theta'$ which is contained in the same element
$\omega$ of $\mathcal{P}_{n_k(\theta)}$ as $\theta$, such
that $|f_{n_k(\theta)}(\theta', X(\theta'))|\le \alpha^{0.6}$. Since
$F^{n_k(\theta)}(X|\omega)$ is an admissible curve, we have that
$f_{n_k(\theta)} (\theta', X(\theta'))\in B_\alpha$ for all $\theta'\in
\omega$. Note that if $n_k(\theta)=n$ and $\omega$ is the element of $\mathcal{P}_n$ which contains $\theta$,
then $n_k(\theta')=n$ for all $\theta'\in \omega$.

For $\alpha>0$ let $N(\alpha)$ be as in
\S\ref{sec:shadowing}.
The main result of this section is the following:

\begin{prop}\label{prop:return0}
Given $M>0$, the following holds provided that
$\alpha>0$ is small enough:  For each $n\ge 0$ and $k\ge 0$, if $\omega$ is an element of $\mathcal{P}_n$ on which $n_k=n$, then
$$\left|\{\theta\in\omega: n_{k+1}(\theta)-n_k(\theta)\le N(\alpha)+M\} \right| \le
\frac{2}{3}|\omega|.$$
\end{prop}

Proposition~\ref{prop:return0} is a direct consequence of
Proposition~\ref{prop:return} which will be proved in
\S~\ref{subsec:rec}. As a corollary of this result, we shall prove
in \S~\ref{subsec:truncate} that suitably truncated vertical
derivatives of $f_n$ at $(\theta, X(\theta))$ is exponentially big
in $n$ for a.e. $\theta$.

\subsection{Recurrence of the critical set}\label{subsec:rec}

\begin{prop}\label{prop:return}
For each $M\in\mathbb{N}$, there exists $\sigma>0$ such that the
following holds provided that $\alpha>0$ is sufficiently small. Let
$X_0:[0,1)\to B_\alpha$ be an admissible  curve
and for $n=0,1,\ldots, $ let
$$\Theta_n=\{\theta\in \T: |f_n(\theta,
X_0(\theta))|<\sigma \}.$$ Then
$$\left|\bigcup_{n=1}^{N(\alpha)+M}\Theta_n\right|\le \frac{2}{3}.$$
\end{prop}

\begin{proof}
Recall that in \S\ref{sec:shadowing}, we proved that there exist constants $C_0\in (0,1)$ and $\xi_0>0$ such that
the following hold:
\begin{itemize}
\item $C_0\le |(f^{N(\alpha)})'(c_1)|\alpha \le C_0^{-1},$
\item for any $(\theta, x)$ with $|x|<\sqrt{\alpha}$, $|f_n(\theta,
x)|\ge \xi_0$ for all $n=1,2,\ldots, N(\alpha)$.
\end{itemize}

Let $\eps=\eps(M)\in (0,1)$ be a small constant so that for any
$x\in [-\eps, \eps]$ and $1\le k\le M$, $|f^{k}(x)|\ge 2\eps$.
Provided that $\alpha>0$ is small enough, it follows that for any
$x\in [-\eps, \eps]$, $\theta\in \T$ and $1\le k\le M$,
we have $|f_k(\theta, x)|\ge \eps$.

Let $N_0$ be a positive integer such that $\phi$ is not of period
$d^{-N_0}$.
Let $N_1\ge N_0$ be a large integer such that
$$e\frac{A}{C_0}  \cdot 4^M |(f^{N_1-1})'(c_1)|^{-1}\le
\frac{\eps}{2}.$$

Let $Y_j(\theta)=f_j(\theta, 0)$, $j=0,1,\cdots$. Then
$$Y_{N_1}(\theta)=f^{N_1}(0)+\alpha Q_{N_1}(\theta) +O(\alpha^2),$$ where
$$Q_{N_1}(\theta)=\sum_{k=1}^{N_1} (f^{N_1-k})'(f^k(0))\phi(d^{k-1}\theta).$$

\noindent {\bf Claim 1.} $Q_{N_1}(\theta)$ is not of period
$d^{-N_0}$.

In fact, let $$\phi(\theta)=c_0+\sum_{n=1}^\infty (c_n e^{2\pi
in\theta}+\overline{c_n} e^{-2\pi in\theta})$$ be the Fourier series
expansion of $\phi$. Since $\phi$ is not of period $d^{-N_0}$, there
exists a minimal positive integer $m_0$ such that $d^{N_0}\not |
m_0$ and $c_{m_0}\not =0$. Then
$$\int_{0}^1 Q_{N_1}(\theta) e^{-2\pi i m_0\theta} d\theta
= (f^{N_1-1})'(f(0))c_{m_0} \not =0,$$ which implies the claim 1.

For any $\theta\in \T$ let
$\theta'=\theta+ d^{-N_1} (\textrm{mod } 1)$. Since $Q_{N_1}$ is
real analytic and is not of period $1/d^{N_1}$, there exists
$\eta=\eta(N_1)>0$ such that
$$\Omega':=\{\theta\in \T: |Q_{N_1}(\theta)-Q_{N_1}(\theta')|\ge 4\eta\}$$
has Lebesgue measure greater than $0.99$. Let $\Omega$ be the union
of all elements of $\mathcal{P}_{N(\alpha)+M} $ which intersect
$\Omega'$. Clearly, $|\Omega|\ge 0.99$. Provided that $\alpha$ is
small enough, $N(\alpha)\gg N_1$, hence the oscillation of
$Q_{N_1}(\theta)$ on any element of $\mathcal{P}_{N(\alpha)+M}$ is
less than $\eta$. Therefore, for each $\theta\in \Omega$,
$$|Q_{N_1}(\theta)-Q_{N_1}(\theta')|\ge 2\eta.$$
It follows that for all $\theta\in \Omega$,
\begin{equation}\label{eqn:yn1}
\eta \alpha\le |Y_{N_1}(\theta)-Y_{N_1}(\theta')|\le A\alpha,
\end{equation}
provided that $\alpha$ is small enough.

By Lemma~\ref{lem:distortion},
\begin{equation}\label{eqn:ynalpha}
\frac{1}{e} |(f^{N(\alpha)-N_1})'(c_{N_1})|\eta \alpha \le
|Y_{N(\alpha)}(\theta) -Y_{N(\alpha)} (\theta')|\le
e|(f^{N(\alpha)-N_1})'(c_{N_1})| A\alpha.
\end{equation}

For all $n\in [N(\alpha), N(\alpha)+M]$, since $|\frac{\partial
f(\theta,x)}{\partial x}|\le 4$ on $I$, we have
\begin{align}\label{eqn:2eps}
|Y_{n}(\theta)-Y_n (\theta')| &
=|f_{n-N(\alpha)}(g^{N(\alpha)}\theta,Y_{N(\alpha)}(\theta))-
f_{n-N(\alpha)}(g^{N(\alpha)}\theta,Y_{N(\alpha)}(\theta'))| \nonumber\\
&\le 4^{n-N(\alpha)}
|Y_{N(\alpha)}(\theta)-Y_{N(\alpha)}(\theta')| \nonumber\\
& \le
eA 4^M|(f^{N(\alpha)-N_1})'(c_{N_1})| \alpha \nonumber\\
& = e \frac{A}{C_0}\cdot 4^M |(f^{N_1-1})' (c_1)|^{-1}\cdot
C_0|(f^{N(\alpha)})'(c_1)|\alpha \nonumber
\\ & \le \eps/2.
\end{align}

Now let $\sigma>0$ be a constant such that
$$\sigma \le\min\left(\frac{\eta}{4e} \frac{C_0}{|(f^{N_1-1})'(c_1)|}
\left(\frac{\eps}{2}\right)^{M}, \frac{\eps}{4}, \frac{\xi_0}{2}\right).$$

{\bf Claim 2.} For each $\theta\in \Omega$, $\theta$ and $\theta'$
does not simultaneously belong to
$$\Theta=\bigcup_{n=1}^{N(\alpha)+M} \{\theta\in [0,1):
|Y_n(\theta)|<2\sigma\},$$
provided that $\alpha>0$ is small enough.

To prove this claim, let us assume that $|Y_m(\theta)|<2\sigma$ for
some $m\in \{1,\ldots, N(\alpha)+M\}$. As $|Y_j(\theta)|\ge \xi_0\ge 2\sigma$ for $1\le j\le N(\alpha)$,
we have $m> N(\alpha)$. Since
$|Y_m(\theta)|\le \eps/2$, by the choice of $\eps$, we obtain that
for $k\in \{N(\alpha),\ldots, m-1\}$, we have $|Y_k(\theta)|\ge
\eps$, hence $[Y_k(\theta),Y_k(\theta')]$ is disjoint from the
region $ [-\eps/2, \eps/2]$ by \eqref{eqn:2eps}. Therefore
\begin{align*}
|Y_m(\theta)-Y_m(\theta')|
&=|f_{m-N(\alpha)}(g^{N(\alpha)}\theta,Y_{N(\alpha)}(\theta))-
f_{m-N(\alpha)}(g^{N(\alpha)}\theta,Y_{N(\alpha)}(\theta'))| \\
&\ge \left(\frac{\eps}{2}\right)^{m-N(\alpha)}
|Y_{N(\alpha)}(\theta)-Y_{N(\alpha)}(\theta')|\ge
\left(\frac{\eps}{2}\right)^{M}\frac{1}{e}
|(f^{N(\alpha)-N_1})'(c_{N_1})|\eta \alpha\\
&= \frac{\eta}{e} \frac{C_0}{|(f^{N_1-1})'
(c_1)|}\left(\frac{\eps}{2}\right)^{M}\cdot
\frac{|(f^{N(\alpha)})'(c_1)|\alpha}{C_0} \ge 4 \sigma,
\end{align*}
hence $|Y_m(\theta')|> 2\sigma$. On the other hand,
$$|Y_m(\theta')|\le |Y_m(\theta)|+|Y_m(\theta)-Y_m(\theta')|<\eps,$$
which implies that $|Y_k(\theta')|\ge \eps$ for all $1\le k\le
N(\alpha)+M$ except for $k=m$. In conclusion, we obtain that
$\theta'\not\in \Theta$. This finishes the proof of Claim 2.

The claim implies that
$$\left|\Theta
\right|\le 0.51.$$

It is clear that there exists $C_2\ge 1$ such that
$|(f^{k_1})'(c_1)|\le C_2 |(f^{k_2})'(c_1)|$ for any $0\le k_1<k_2$.
For any $\theta\in \Theta_n$, $n=1,2,\ldots, N(\alpha)+M$, we have
\begin{align*}
|f_n(\theta,0)-f_n(\theta,X_0(\theta))|&=|f_{n-1}(g\theta,f_1(\theta,0))-
f_{n-1}(g\theta,f_1(\theta,X_0(\theta)))|\\
&=|\frac{\partial f_{n-1}(g\theta,\xi)}{\partial x}|\cdot |X_0(\theta
)|^2
\text{ for some $\xi\in [f_1(\theta,X_0(\theta )),f_1(\theta,0)]$}\\
&\le \begin{cases} e|(f^{n-1})'(c_1)|\alpha^{1.2}, \ \ \ \ \ \ &n-1\le N(\alpha)\\
4^{n-1-N(\alpha)}e|(f^{N(\alpha)})'(c_1)|\alpha^{1.2}, \ \ \ \ \ \
&n-1\ge N(\alpha)
\end{cases}\\
&\le C_2e4^{M}  |(f^{N(\alpha)})'(c_1)| \alpha ^{1.2} \le  \frac{
4^MC_2 e}{C_0} \alpha ^{0.2} \le \sigma,
\end{align*}
 provided that $\alpha$ is
small enough. Hence $|f_n(\theta, 0)|<2\sigma$. This proves that
$\bigcup_{n=1}^{N(\alpha)+M}\Theta_n\subset \Theta$, completing the
proof of the proposition.
\end{proof}

\subsection{Truncated vertical partial derivative}
\label{subsec:truncate}
Consider an admissible curve $X:[0,1)\to B_\alpha$. Define $n_k(\theta)$ as above.
For each $k=1,2,\ldots$ and $M>0$, define
$$\mathcal{B}_k(M):=\{\theta\in [0,1): n_k(\theta)\le (N(\alpha)+M) k\}.$$
The following is an easy consequence of Proposition~\ref{prop:return0}.

\begin{cor}\label{cor:nklarge}
Given $M>0$, the following holds provided that $\alpha>0$ is small enough.
For each $k=1,2,\ldots$,
we have $|\mathcal{B}_k(M)|\le 0.8^k.$ In particular, for a.e. $\theta\in
[0,1)$, $n_k(\theta)> (N(\alpha) + M) k$ for all $k$ sufficiently
large.
\end{cor}

\begin{proof} Let $L>0$ be a large number such that
\begin{equation}\label{eqn:L}
\sum_{(1-L^{-1})k\le t\le k}\binom{k}{t}\left(\frac{2}{3}\right)^t\le 0.8^k
\end{equation}
holds for all positive integers $k$.

For any finite sequence of integers $0\le k_1<k_2<\cdots< k_t$, let
$$\mathcal{D}(k_1,k_2,\cdots, k_t)=\{\theta\in[0,1): n_{k_i+1}(\theta)-n_{k_i}(\theta)\le N(\alpha) +LM, 1\le i\le t\}.$$
We claim that
\begin{equation}\label{eqn:calD}
|\mathcal{D}(k_1,k_2,\cdots, k_t)|\le (2/3)^t,
\end{equation} provided that $\alpha>0$ is small enough.

To prove the claim,  let $\mathcal{F}_k^n$, $k\ge 0$, $n\ge 0$, denote the collection of elements of $\mathcal{P}_n$ on which $n_k(\theta)=n$, and let $\mathcal{F}_k=\bigcup_{n=0}^\infty \mathcal{F}_k^n$.
Let $D_k$ denote the union of elements of $\mathcal{F}_k$ which intersects $\mathcal{D}(k_1,k_2,\cdots,k_t)$. Clearly $D_{k+1}\subset D_k$ for all $k\ge 0$. By Proposition~\ref{prop:return0} (with $M$ replaced by $LM$), for each $1\le i\le t$, and each element $\omega$ of $\mathcal{F}_{k_i}$, we have
$|D_{k_i+1}\cap \omega |\le \frac{2}{3}|\omega|$, provided that $\alpha>0$ is small enough. Thus $|D_{k_i+1}|\le (2/3) |D_{k_i}|$.
Since $\mathcal{D}(k_1,k_2,\ldots, k_t)\subset D_{k_t+1}$, the inequality (\ref{eqn:calD}) follows.

Now let us complete the proof of the corollary. If $\theta\in\mathcal{B}_k(M)$, then
$$\sum_{j=0}^{k-1} (n_{j+1}(\theta)-n_j(\theta)-N(\alpha))\le kM.$$
As $n_{j+1}(\theta)-n_j(\theta)\ge N(\alpha)$ holds for all $j$, it follows that
$$\#\{0\le j< k: n_{j+1}(\theta) -n_j(\theta) > LM+N(\alpha)\}\le \frac{k}{L}.$$
Thus there exist integers $t\ge (1-L^{-1})k$ and $0\le k_1<k_2<\cdots<k_t$ such that $\theta\in \mathcal{D}(k_1,k_2,\ldots, k_t)$. By (\ref{eqn:calD}), $|\mathcal{B}_k(M)|$ is bounded from above by the left hand side of (\ref{eqn:L}). The corollary is proved.
\end{proof}

For $(\theta, x)\in \T\times \R$, and $n\ge 1$, define
\begin{equation}
 \frac{\hat \partial f_n (\theta, x)}{\partial x}=\prod_{i=0}^{n-1} \max \left(
2|f_i(\theta, x)|, 2\alpha\right).
\end{equation}

\begin{prop}\label{prop:hatd}
There exists $\lambda_1>1$ such that for each positive integer $M$,
the following holds provided that $\alpha$ is small enough. Let $X:
[0,1)\to B_\alpha$ be an admissible curve. Then for a.e. $\theta\in
[0,1)$,
$$\left|\frac{\hat \partial f_n (\theta, X(\theta))}{\partial x}\right|\ge \lambda_1^{nM/N(\alpha)}.$$
holds for all $n$ large.
\end{prop}
\begin{proof} Let $K>0$ and
$\lambda>1$ be as in Lemma~\ref{lem:vianaexpansion} and let $C_0$ be
as in Lemma~\ref{lem:expalpha}. Take $M_*\in \mathbb{N}$ with
$\lambda^{M_*}KC_0\ge 1$. Define $n_1, n_2, \ldots$ as above and fix
$M$. We only need to consider those $\theta$ for which $n_k(\theta)$
is defined for all $k\ge 0$. By Corollary~\ref{cor:nklarge}, for a.e. $\theta\in
[0,1)$, $n_k(\theta)\ge k(N(\alpha)+ M+M_*+1)$ for all $k$ sufficiently large. Let us fix such a
$\theta$ and prove that the expansion estimate holds.
Take $L=L(\theta)>0$ such that  $n_k(\theta)\ge k(N(\alpha)+ M+M_*+1)$ for all $k\ge L$.
Write
$z_i=F^i(\theta, X(\theta))$ and $n_i=n_i(\theta)$. By
Lemma~\ref{lem:vianaexpansion},
\begin{align}\label{eq-gj0}
|\frac{\partial
f_{n-n_k-1}(z_{n_k+1})}{\partial x}|\ge
K\alpha^{0.6}\lambda^{n-n_k-1}
\end{align} for $n\in [n_k, n_{k+1}]$ and
\begin{align}\label{eq-gj1}
\left|\frac{\partial
f_{n_{k+1}-n_k-N(\alpha)-1}(z_{n_k+N(\alpha)+1})}{\partial x}
\right| \ge K \lambda^{n_{k+1}-n_k-1 -N(\alpha)}.
\end{align}
By Lemma~\ref{lem:expalpha},
\begin{align}\label{eq-gj2}
\left|\frac{\partial f_{N(\alpha)} (z_{n_k+1})}{\partial
x}\right|\ge \frac{C_0}{\alpha}.
\end{align}
By \eqref{eq-gj1} and \eqref{eq-gj2},
\begin{align}\label{eq-gj3}
\left|\frac{\partial f_{n_{k+1}-n_k-1} (z_{n_k+1})}{\partial
x}\right|\ge \frac{C_0K}{\alpha}\lambda^{n_{k+1}-n_k-1 -N(\alpha)}.
\end{align}
 Thus for $n_k\le n< n_{k+1}$ with $k\ge L$, using \eqref{eq-gj0} and \eqref{eq-gj3} we have
\begin{align*}
\left|\frac{ \hat \partial f_n (\theta, X(\theta))}{\partial
x}\right|&=\left|\frac{  \partial f_n (\theta, X(\theta))}{\partial
x}\right|\cdot \prod_{i=0}^k\frac{\alpha}{|f_{n_i}(\theta,X(\theta))|}\\
&= \alpha^{k+1} \left|\frac{ \partial f_{n-n_k-1}
(z_{n_k+1})}{\partial x}\right|\cdot \prod_{i=0}^{k-1}
\left|\frac{\partial f_{n_{i+1}-n_i-1} (z_{n_i+1})}{\partial
x}\right|
\\ &
\ge \alpha^{1.6} K^{k+1} C_0^k
\lambda^{n-kN(\alpha)-(k+1)}=\alpha^{1.6} K \lambda^{-1}\left(\frac{KC_0}{\lambda^{N(\alpha)+1}}\right)^k
\lambda^n\\
&\ge \alpha^{1.6} K\lambda^{-1}
\left(\frac{KC_0}{\lambda^{N(\alpha)+1}}\right)^{\frac{n}{N(\alpha)+M+M_*+1}}
\lambda^n   \ \ \ \  \text{(when $\alpha$ is small enough)}\\
&= \alpha^{1.6} K\lambda^{-1} (KC_0\lambda^{M_*}
\lambda^M)^{\frac{n}{N(\alpha)+M+M_*+1}}\ge
\alpha^{1.6} K\lambda^{-1} \cdot \lambda^{\frac{nM}{N(\alpha)+M+M_*+1}} \\
&\ge \alpha^{1.6} K\lambda^{-1} \cdot
\lambda^{\frac{nM}{2N(\alpha)}} \ \ \ \ \ \ \ \ \ \ \ \ \ \ \ \ \ \
\ \ \ \ \ \ \ \ \ \text{(when $\alpha$ is small enough)}.
\end{align*}
Taking $\lambda_1=\lambda^{\frac{1}{3}}$  gives us the desired
estimate.
\end{proof}

\section{Exclusion of bad values}
We analyze returns of points to the region $\T\times (-\alpha,
\alpha)$. Using a large deviation argument from \cite{V}, we obtain
bounds for such deep returns which, together with the estimate given
by Proposition~\ref{prop:hatd}, shows that $F$ has a positive
vertical Lyapunov exponent almost everywhere.

 Write
$$\chi_{-}(\theta, x)=\liminf_{n\rightarrow +\infty} \frac{1}{n} \log \left|\frac{\partial f_n(\theta, x)}{\partial
x}\right|.$$

\begin{prop}\label{prop:key}
Let  $X: [0,1)\to B$ be an admissible curve. Then for a.e.
$\theta\in [0,1)$, $\chi_{-}(\theta, X(\theta))>0$.
\end{prop}

\begin{proof}[Proof of the Theorem \ref{thm:main}.]
It suffices to prove that $\chi_{-}(\theta, x)>0$ for a.e. $(\theta,
x)\in \T\times \R$. Since $|\frac{
\partial f_1(\theta,x)}{\partial x}|>\beta>1$ outside of $I$, we only need to consider
$(\theta, x)\in \T\times B$.
By Fubini's theorem, we only need to show that for any $x\in
B\setminus \{0\}$, the set $$Q(x)=\{\theta\in \T: \chi_{-}(\theta,
x)\le 0\}$$ has measure zero. Let $n_0$ be as in Lemma
\ref{lem:lyapunov} and let $\omega$ be an element of
$\mathcal{P}_{n_0+1}$. Then by Lemma \ref{lem:horizontal},
$\theta\mapsto f_{n_0+1}((g^{n_0+1}|\omega)^{-1}(\theta),x)$ defines an
admissible curve. By Proposition~\ref{prop:key}, we have
$$\chi_{-}(F^{n_0+1}(\theta, x))>0$$ for a.e. $\theta\in \T.$
Clearly, for each $j=0,1,\ldots, n_0$, there are only finitely
many $\theta$ for which $f_j(\theta, x)=0$. Thus $\chi_{-}(\theta,
x)=\chi_{-}(F^{n_0+1}(\theta, x))>0$ for a.e. $\theta\in \T$.
\end{proof}

Let us turn to the proof of Proposition~\ref{prop:key}. Without loss
of generality, let us consider an admissible curve $X:[0,1)\to
B_\alpha$, where $B_\alpha= [-2\alpha^{0.6}, 2\alpha^{0.6}]$ as
before. Let
$$Q:=\{\theta\in [0,1): f_i(\theta, X(\theta))=0\mbox{ for some
}i=0,1,\ldots\},$$ which is a countable set.
As in \S~\ref{sec:recurrence}, we define
$0=n_0(\theta)<n_1(\theta)<n_2(\theta)<\cdots$.


Let $\eps_0>0$ be as in Proposition~\ref{prop:deepreturn}. Let
$\eps_1\in (0, \eps_0)$ be a small constant and let
\begin{equation}\label{eqn:delta}
\Delta=\Delta(\eps_1):=\log_d \frac{1}{\eps_1}.
\end{equation}
For any $\theta\in [0,1)\setminus Q$ and $k\ge 0$, if $n_k(\theta)$
is defined, then we define $\hat q_k(\theta)$ to be the integer such
that
$$d^{-\hat q_k(\theta)+1} \alpha> |f_{n_k(\theta)}(\theta, X(\theta))|\ge d^{-\hat q_k(\theta)}\alpha.$$
Moreover, define $q_k(\theta)=0$ if $d^{-{(\hat q_k(\theta)-1)}}\ge
\eps_1$ and $q_k(\theta)=\hat q_k(\theta)$ otherwise. If
$n_k(\theta)$ is not defined, then we set $q_k(\theta)=0$.

Therefore $q_k(\theta)>0$ implies that
\begin{equation}\label{eqn:qkdelta}
q_k(\theta)> \Delta.
\end{equation}

For $c>0$ and $K\ge 1$, let
$$E_K(c):=\{\theta\in [0,1)\setminus Q: \sum_{k=0}^K
q_k(\theta) \ge cK\}.$$

\begin{lem} \label{lemma:8}For any $c>0$, the following holds for all $K$ large:
 $$|E_K(c)|\le d^{-\frac{\kappa}{6} c \sqrt{K}},$$
 provided that $\eps_1$ was chosen small enough,
 where $\kappa$ is as in Proposition
\ref{prop:deepreturn}.
\end{lem}
\begin{proof}
We decompose the set $$E_K(c)=E_K^1(c)\cup E_K^2 (c),$$ where
\begin{align*}
E_K^2(c)& :=\{\theta\in E_K(c):\max \limits_{0\le k \le K}
q_k(\theta) \ge
[\sqrt{K}]\},\\
E_K^1(c)& := E_K(c)\setminus E_K^2(c).
\end{align*}
 By
Proposition~\ref{prop:deepreturn},
\begin{equation}\label{eqn:ek2c}
|E_K^2(c)|\le \sum_{k=0}^K |\{\theta\in E_K(c): q_k(\theta)\ge
[\sqrt{K}]\}|\le (K+1) d ^{-\kappa([\sqrt{K}]-1)}.
\end{equation}
Let us estimate $|E_K^1 (c)|$. To this end, let $m=[\sqrt{K}]$ and
let
$$\mathcal{A}_p=\{0\le k\le K: k\equiv p \mod m\}.$$
Let
$$E_{K,p}^1(c):= \left\{\theta\in E_K^1(c): \sum_{k\in \mathcal{A}_p} q_k(\theta)\ge
\frac{cK}{m}\right\}.$$ Then
$$E_K^1(c)=\bigcup_{p=0}^{m-1} E_{K, p}^1(c).$$

Let us now fix $p$, and estimate $|E_{K, p}^1(c)|$. For each
$\textbf{r}=(r_k)_{k\in \mathcal{A}_p}\in \{0,1,\ldots, [\sqrt{K}]
\}^{\mathcal{A}_p}$, let $\|\textbf{r}\|=\sum_k r_k$, and let
$$E_{K, p}^1(c,\textbf{r})=\{\theta\in E_{K, p}^1(c):
q_k(\theta)=r_k \mbox{ for all } k\in\mathcal{A}_p\}.$$

{\bf Claim.} For each $\textbf{r}\in \{0,1,\ldots, [\sqrt{K}]
\}^{\mathcal{A}_p}$, we have
\begin{equation}\label{eqn:ekpcr}
|E_{K, p}^1(c,\textbf{r})|\le d^{-\frac{\kappa }{2} \|\textbf{r}\|}
\end{equation}

\begin{proof}[Proof of Claim.]
Let $k_1<k_2<\ldots< k_\nu$ be all the elements in $\mathcal{A}_p$
such that $r_{k_j}>0$.

For each $\theta\in E_{K, p}^1(c,\textbf{r})$, let
$\omega'_j(\theta)$ (resp. $\omega_j(\theta)$) be the element of
$\mathcal{P}_{n_{k_j}(\theta)}$ (resp.
$\mathcal{P}_{n_{k_j}(\theta)+r_{k_j}}$) which contains $\theta$.
Let $$\Omega_j':=\bigcup_{\theta\in E_{K, p}^1(c, \textbf{r})}
\omega_j'(\theta), \,\,\,\,\, \Omega_j:=\bigcup_{\theta\in E_{K,
p}^1(c, \textbf{r})} \omega_j(\theta).$$ Since
$$n_{k_{j+1}}(\theta)-n_{k_j}(\theta)-r_{k_j}\ge mN(\alpha)-r_{k_j}\ge 0,$$ holds for all
$\theta$,
$\Omega_{j+1}'\subset \Omega_j$.

For each $\omega_j'(\theta)$, since $F^{n_{k_j}}(X|\omega_j')$ is an
admissible curve, we have that
$$|f_{n_{k_j}}(\theta', X(\theta'))|\le |f_{n_{k_j}}(\theta, X(\theta))|+ A\alpha d^{-r_{k_j}}
\le (d+A)\alpha d^{-r_{k_j}}, $$ holds for all  $\theta'\in
\omega_j'(\theta)\cap \Omega_j$. By
Proposition~\ref{prop:deepreturn}, it follows that
$$\frac{\left|\omega_j'(\theta)\cap
\Omega_j\right|} {|\omega'_j(\theta)|}\le (d+A)^\kappa d^{-\kappa r_{k_j}}\le
d^{-\kappa r_{k_j}/2},$$ provided that we have chosen $\eps_1$ small
enough in \eqref{eqn:delta} so that $r_{k_j}\ge \Delta$ is large.

Thus for each $j$, $$|\Omega_j|\le d^{-\kappa r_{k_j}/2}
|\Omega_j'|,$$ hence
$$|\Omega_{j+1}|\le d^{-\kappa r_{k_{j+1}/2}} |\Omega_j|.$$
It follows that
$$|E_{K, p}^1(c,\textbf{r})|\le |\Omega_\nu|
\le d^{-\frac{\kappa}{2} \sum \limits_{k\in \mathcal{A}_p} r_{k}}.$$
\end{proof}

To complete the proof of the lemma, let us estimate the number $Q(K,
p, R)$ of $\textbf{r}=(r_k)\in \{0,1,\ldots,
[\sqrt{K}]\}^{\mathcal{A}_p}$ with $\sum_{k\in\mathcal{A}_p} r_k=R$
and with $E_{K,p}^1(c,\textbf{r})\not=\emptyset$. Clearly,
$$Q(K, p, R)\le \sum_{\nu=1}^{\min(\nu_p, [R/\Delta])}
\left(\begin{array}{ll}
\nu_p \\
\nu
\end{array}\right)
\left(
\begin{array}{ll}
R-\nu [\Delta]\\
 \nu-1
\end{array}
\right)
,
$$
where $\nu_p=\# \mathcal{A}_p\approx \sqrt{K}$. The first binomial
coefficient corresponds to the number of possible positions for
which $r_k\not=0$, and the second corresponds to the distribution of
the sum into terms. Assuming $\eps_1$ small so that $\Delta$ large,
let us prove that
\begin{equation}\label{eqn:qkpr}
Q(K, p, R)\le d^{\kappa R/4}
\end{equation}
holds for all $K$ sufficiently large and $R\ge cK/ m$.

To this end, by Stirling's formula we first observe that for each $1\le \nu\le
[R/\Delta]$,
$$\left(
\begin{array}{ll}
R-\nu [\Delta]\\
 \nu-1
\end{array}
\right) \le \left(
\begin{array}{ll}
R\\
\nu
\end{array}
\right) \le \left(
\begin{array}{ll}
 R \\
 \left[R/\Delta\right]
\end{array}
\right)\le d^{\kappa R/8}
$$
for $R$ sufficiently enough, provided that $\eps_1$ small . Thus
\begin{align*}
Q(K, p, R) \le \sum_{\nu=1}^{\nu_p} \left(\begin{array}{ll} \nu_p \\
\nu
\end{array}\right)d^{\kappa R/8}  \le 2^{\nu_p} d^{\kappa R/8}.
\end{align*}
In particular, there is a constant $C$ such that (\ref{eqn:qkpr})
holds if $R>C\sqrt{K}$.  Assume that $R\le C\sqrt{K}$. Provided that
$\Delta$ is large enough, $[R/\Delta]< \nu_p$, so
\begin{align*}
Q(K, p, R) \le d^{\kappa R/8} \sum_{\nu=1}^{[R/\Delta]}
\left(\begin{array}{ll}
\nu_p \\
\nu
\end{array}\right)
 \le d^{\kappa R/8} [R/\Delta]
\left(\begin{array}{ll}
\nu_p \\
\left[ R/\Delta\right]
\end{array}\right).
\end{align*}
Since $\nu_p\approx \sqrt{K}$ and $R\ge cK/m \approx c \sqrt{K}$,
(\ref{eqn:qkpr}) follows.

By the Claim, we obtain
\begin{align*}
|E_{K,p}^1(c)|&=\sum_{\textbf{r}\in \{0,1,\ldots, [\sqrt{K}]\}^{\mathcal{A}_p}:\,
\parallel \textbf{r}\parallel\ge \frac{cK}{m} } |E_{K, p}^1(c, \textbf{r})|
\le \sum_{\textbf{r}\in \{0,1,\ldots, [\sqrt{K}]\}^{\mathcal{A}_p}:\,
\parallel \textbf{r}\parallel\ge \frac{cK}{m} } d^{-\frac{\kappa}{2}\parallel \textbf{r}\parallel} \\
&
\le \sum_{R\ge cK/m} Q(K,p,R) d^{-\frac{\kappa}{2}R}\le \sum_{R\ge cK/m}  d^{-\kappa R/4}\le \frac{1}{1-d^{-\frac{\kappa}{4}}} d^{-\kappa c
\sqrt{K}/4}.
\end{align*}
 It follows that
\begin{equation}\label{eqn:ek1c}
|E_K^1(c)|\le \sum \limits_{p=0}^{m-1} |E_{K, p}^1(c)|\le  m \frac{1}{1-d^{-\frac{\kappa}{4}}} d^{-\kappa c\sqrt{K}/4}\le d^{-\kappa
c\sqrt{K}/5}
\end{equation}
 when $K$ large. Combining (\ref{eqn:ek2c}) and (\ref{eqn:ek1c}) completes the proof
of the lemma.
\end{proof}
\begin{proof}[Proof of Proposition~\ref{prop:key}]
By Lemma \ref{lemma:8}, fix a small constant $\eps_1>0$ so that $\sum_{K=1}^\infty |E_K(1)|$
converges. Hence, for a.e. $\theta\in [0,1)\setminus Q$, there exists
$K_0=K_0(\theta)$ such that $\theta \not\in E_K(1)$ for all $K\ge
K_0$.

Choose a large positive integer $M$ such that $\lambda'= \lambda_1^M
\eps_1 d^{-2}
>1$, where $\lambda_1>1$ is as in Proposition~\ref{prop:hatd}. By Proposition~\ref{prop:hatd}, provided that $\alpha>0$ is
small enough, the following holds: for a.e. $\theta\in [0,1)$,
$$\frac{\hat \partial f_n(\theta, X(\theta))}{\partial x}\ge \lambda_1^{\frac{Mn}{N(\alpha)}}$$
holds for all large $n$.

Now take $\theta$ with the above properties. For any $n\ge 1$, and
let $k_n$ be maximal such that $n_{k_n}(\theta)\le n$. Then $k_n\le
n/N(\alpha)$ and for all large $n$,
\begin{align*}
\left|\frac{\partial f_n(\theta, X(\theta))}{\partial x}\right|&
=\frac{\hat
\partial f_n (\theta, X(\theta))}{\partial x} \prod_{k=0}^{k_n}
\frac{|f_{n_k}(\theta,
X(\theta))|}{\alpha}\\
& \ge \frac{\hat \partial f_n(\theta, X(\theta))}{\partial x}
\prod_{k=0}^{k_n} \min \left(\frac{\eps_1}{d},
d^{-q_k(\theta)}\right)\\
& \ge \lambda_1^{nM/ N(\alpha)} \left(\frac{\eps_1}{d}\right)^{k_n+1}
d^{-\sum \limits_{k=0}^{k_n} q_k(\theta)}.
\end{align*}
We may assume that $k_n\to\infty$, for otherwise, $\chi_{-}(\theta,
X(\theta))>0$  obviously holds.  For $n$ large, $k_n\ge K_0$, hence
$$\sum_{k=0}^{k_n} q_k(\theta)\le k_n.$$
It follows that for $n$ large,
$$\left|\frac{ \partial f_n(\theta, X(\theta))}{\partial x}\right|\ge \lambda_1^{nM/N(\alpha)} \eps_1
^{k_n+1} d^{- (2 k_n+1) } \ge \frac{\eps_1}{d}
\left(\lambda_1^{M}\eps_1 d^{-2}\right)
^{n/N(\alpha)}=\frac{\eps_1}{d} (\lambda')^{n/N(\alpha)}.$$
Therefore $\chi_{-}(\theta, X(\theta))>0$.
\end{proof}

{\bf Acknowledgments}.
The first author was partially supported by NSFC (No. 11225105), Fok Ying Tung Education Foundation and the Fundamental
Research Funds for the Central Universities WK0010000014. The second author is partially supported by Grant C-146-000-032-001 from National University of Singapore. Both authors thank the anonymous referee for his careful reading and valuable comments.

\end{document}